\documentclass[reqno,12pt]{amsart}
\usepackage{latexsym,amssymb,amsmath,epsfig}
\usepackage{amsthm}
\usepackage[latin1]{inputenc} 
\newtheorem{theorem}{Theorem}[section]
\newtheorem{lemma}{Lemma}[section]

\newtheorem{proposition}{Proposition}[section]

\newtheorem{remark}{Remark}[section]

\newcommand{\subgp}[1]{\langle{#1}\rangle}
\begin{document}
\title{Characterizing path-like trees from linear configurations}
\author{S. C. L\'opez}
\address{%
Departament de Matem\`{a}tiques\\
Universitat Polit\`{e}cnica de Catalunya\\
C/Esteve Terrades 5\\
08860 Castelldefels, Spain}
\email{susana.clara.lopez@upc.edu}

\author{F. A. Muntaner-Batle}
\address{Graph Theory and Applications Research Group \\
 School of Electrical Engineering and Computer Science\\
Faculty of Engineering and Built Environment\\
The University of Newcastle\\
NSW 2308
Australia}
\email{famb1es@yahoo.es}

\maketitle

\begin{abstract}
Assume that we embed the path $P_n$ as a subgraph of a $2$-dimensional grid, namely, $P_k \times P_l$. Given such an embedding, we consider the ordered set of subpaths
$L_1, L_2, \ldots , L_m$ which are maximal straight segments in the embedding, and such that
the end of $L_i$ is the beginning of $L_{i+1}$. Suppose that $L_i\cong P_2$, for some $i$ and that
some vertex $u$ of $L_{i-1}$ is at distance $1$ in the grid to a vertex $v$ of $L_{i+1}$. An elementary
transformation of the path consists in replacing the edge of $L_i$ by a new edge $uv$. A tree $T$ of order $n$ is said to be a path-like tree, when it can be obtained from some
embedding of $P_n$ in the $2$-dimensional grid, by a sequence of elementary transformations. Thus, the maximum degree of a path-like tree is at most $4$.

Intuitively speaking, a tree admits a linear configuration if it can be described by a sequence of paths in such a way that only vertices from two consecutive paths, which are at the same distance of the end vertices are adjacent.
In this paper, we characterize path-like trees of maximum degree $3$, with an even number of vertices of degree $3$, from linear configurations. We also show that the characterization of path-like trees of maximum degree $4$ can be reduced to the characterization of path-like tree of maximum degree $3$.\end{abstract}

\begin{quotation}
\noindent{\bf Key Words}: {path-like tree, normalized embedding, linear configuration}

\noindent{\bf 2010 Mathematics Subject Classification}:  Primary 05C05, 05C75,
   Secondary 05C70 and 05C78
\end{quotation}
\section{Introduction}
For the graph theory terminology and notation not defined in this paper, we refer the reader to one of the following
sources \cite{CH,G,W}.

Path-like trees were first introduced by Barrientos in \cite{BarThesis} as an alternative way to attack the well known graceful tree conjecture \cite{Rosa}. It turns out that strong connections of path-like trees with other well known labeling conjectures, as for instance, the harmonious tree conjecture introduced in 1980 by Graham  and Sloane \cite{GrSlo} or the super edge-magic conjecture by Enomoto et al. \cite{E} have also been found. Of course, due to its relation with graceful labelings, path-like trees have also relations with graph decomposition problems. See \cite{Barat, BHLMT, Botler, Thom} for recent advances on decompositions of graphs by trees.
 Path-like trees were defined as follows: assume that we embed the path $P_n$ as a subgraph of a $2$-dimensional grid, namely, $P_k \times P_l$. Given such an embedding, we consider the ordered set of subpaths
$L_1, L_2, \ldots , L_m$ which are maximal straight segments in the embedding, and such that
the end of $L_i$ is the beginning of $L_{i+1}$. Suppose that $L_i\cong P_2$, for some $i$ and that
some vertex $u$ of $L_{i-1}$ is at distance $1$ in the grid to a vertex $v$ of $L_{i+1}$. An {\it elementary
transformation} of the path consists in replacing the edge of $L_i$ by a new edge $uv$. We
say that a tree $T$ of order $n$ is a {\it path-like tree}, when it can be obtained from some
embedding of $P_n$ in the $2$-dimensional grid, by a sequence of elementary transformations.

Although Barrientos introduced path-like trees due to their labeling properties, it is also true that this set of trees is also of interest in its own right.
For instance, a very interesting problem in relation to this topic is to develop an algorithm to determine when a given tree is a path-like tree.
Ba\v{c}a et al. \cite{BaLinMun07} have even gone further in this question and asked for the time complexity of determining when a given
tree is a path-like tree (see also \cite{BaMi}).
However, up to this point, the number of results that allow us to determine when a given tree is a path-like tree is very limited.
In general, the results found in this direction deal with trees that are of very specific types.
For instance, in \cite{MunRiu08} we find characterizations of path-like trees with degree sequence $1,1,1,2,2,\ldots, 2,3$ and degree sequence
$1,1,1,1,2,2,\ldots, 2,4$. Also in the same paper, it was defined the concept of expandable tree, and there were provided different characterizations
for expandable trees. It turns out that every path-like tree is an expandable tree. Thus, if we are able to determine that a tree is not 
expandable, then such a tree cannot be a path-like tree. Using this fact, Muntaner-Batle and Rius-Font were able to obtain results of the following type.

\begin{theorem}\cite{MunRiu08}
Let $T$ be a path-like tree. Then, there are at most two vertices $u,v\in V(T)$ with:
\begin{itemize}
\item[(i)] The degree of $u$ and $v$ is $3$.
\item[(ii)] If the neighborhood of $u$ is $\{u_1,u_2,u_3\}$ and the neighborhood of $v$ is $\{v_1,v_2,v_3\}$ then deg$(u_i)=$deg$(v_i)=1$, for
$i=1,2$ and deg$(u_3)=$deg$(v_3)=2$.
\end{itemize}
\end{theorem}

It is worth mentioning that the previous theorem appeared first in \cite{BaLinMun07} and it was proven without the aid
of expandable trees. However, this original proof is harder to follow and less elegant than the proof provided in \cite{MunRiu08}.

The main goal of this paper is to provide a characterization that will allow us to determine when a tree of maximum degree $3$, with an even number of vertices of degree $3$ (and a small restriction) is a path-like tree, 
using divisibility conditions. We belief that this is a big break thru into the problem since it is an easy consequence of the definition
of path-like trees, that all path-like trees have maximum degree at most $4$. Furthermore, we also feel that similar techniques
 may lead to a complete characterization of other families of path-like trees.


Before concluding this introduction, we will introduce the notion of normalized embedding of a path-like tree, found in \cite{BaLinMun09}, since
it will be of help to better understand the coming material of the paper.

\subsection{Normalized embedding}
Let $\mathbb{L}$ be the $2$-dimensional lattice. If we fix a crossing point as $(0, 0)$ then each
crossing point in $\mathbb{L}$ is perfectly determined by an ordered pair $(i, j)$ where $i$ denotes
the column and $j$ denotes the row of $\mathbb{L}$. Consider an embedding of a path $P$ in
 $\mathbb{L}$ satisfying the following conditions:

 \begin{itemize}
   \item One end vertex of the path $P$ is $(0, 0)$.
   \item Each row of the embedding contains at least two vertices of the path $P$, and each
vertical subpath is in the embedding isomorphic to $P_2$.
   \item Assume that $j$ is an even integer and that $(i, j)$, $(i+1,j)$, $(i+ 2, j )$,\ldots, $(i+ s, j )$
is a maximal straight horizontal subpath in the embedding of the path $P$ in $\mathbb{L}$.
If $(m,j+1)$ belongs to the embedding of the path $P$ in $L$, then $m \le i+s$.
\item Assume that $j$ is an odd integer and that $(i, j)$, $(i-1,j)$, $(i- 2, j )$,\ldots, $(i- t, j )$
is a maximal straight horizontal subpath in the embedding of the path $P$ in $L$.
If $(m,j+1)$ belongs to the embedding of the path $P$ in $\mathbb{L}$, then $m \ge i-t$.
 \end{itemize}

Then, this embedding is called a {\it normalized embedding} of $P$ in the lattice $\mathbb{L}$.

In \cite{BaLinMun09}, it was proven the following result.
\begin{lemma}\cite{BaLinMun09}\label{lemma: every plt from norma embb}
Every path-like tree can be obtained from a normalized embedding of a path in the $2$-dimensional grid.
\end{lemma}

In this paper,  we present a different way of introducing path-like trees, that is, using the notion of a linear configuration. Intuitively speaking, a tree admits a linear configuration if it can be described by a sequence of paths in such a way that only vertices from two consecutive paths, which are at the same distance of the end vertices are adjacent. The formal definition is given in Section \ref{section: A different way of introducing path-like trees}, where we also introduce the trees of type $H$.
In this paper, we characterize path-like trees of maximum degree $3$, with an even number of vertices of degree $3$, from linear configurations and decompositions into trees of type $H$. The main result of the paper is Theorem \ref{theo: a tree in H_n such that...Then path-like tree_good} in Section 3. The characterization of path-like tree of maximum degree $4$ can be reduced to the characterization of path-like trees of maximum degree $3$, as we show in Section \ref{section: degree 4}.
\section{A different way of introducing path-like trees}\label{section: A different way of introducing path-like trees}
Muntaner-Batle and Rius-Font introduced in \cite{MunRiu08} the concept of generalized path-like trees. We follow their idea in order to introduce a different way to understand path-like trees.
Assume that we have a set of paths $P^1_{k_1},P^2_{k_2}, \ldots, P^n_{k_n}$ with
$$V(P^i_{k_i})=\{v^i_1,v^i_2,\ldots, v^i_{k_i}\} \hbox{ and } E(P^i_{k_i})=\{v^i_1v^i_2,v^i_2v^i_3,\ldots, v^i_{k_i-1}v^i_{k_i}\},$$
for $i=1,2,\ldots,n$,
and assume that we embed the set of paths in a horizontal line as follows:
$$v^1_1v^1_2\ldots v^1_{k_1}\hspace{0.5cm} v^2_1v^2_2\ldots v^2_{k_2}\hspace{0.5cm}\ldots \hspace{0.5cm}v^n_1v^n_2\ldots v^n_{k_n}.$$
Let $T$ be any tree that can be obtained from this embedding by joining two consecutive paths in such a way that, for each $l=1,2,\ldots,n-1$,
\begin{equation}\label{eq: condi_union_consec_paths}
   \hbox{if } v^l_iv^{l+1}_j\in E(T) \hbox{ then } d_{T}(v^l_i,v^l_{k_l})=d_{T}(v^{l+1}_1,v^{l+1}_j) \hbox{ and } j>1.
\end{equation}

\begin{theorem}\label{theo: plt iff a linear config. exists}
A tree $T$ can be drawn in the way described above if and only if $T$ is a path-like tree.
\end{theorem}

\begin{proof}
Assume that a tree $T$ can be drawn in the way described above. It follows that $T$ can be drawn in such a way that all vertices of $T$ form the components of a linear forest embedded in  a horizontal line. Let the components of this lineal forest be:   
$$v^1_1v^1_2\ldots v^1_{k_1}; v^2_1v^2_2\ldots v^2_{k_2}; \ldots ; v^n_1v^n_2\ldots v^n_{k_n},$$
where the component generated by $v^i_1,v^i_2,\ldots, v^i_{k_i}$ is denoted by $P^i$, for $i\in\{1,2,\ldots,n\}$. Then $P^1,P^2,\ldots,P^n$ are embedded in a horizontal line and the remaining edges of $T$ join vertices of consecutive paths $P^i, P^{i+1}$, where $i\in \{1,2,\ldots,n-1\}$. Furthermore, if $v_i^lv_j^{l+1}\in E(T)$ then $d_{T}(v^l_i,v^l_{k_l})=d_{T}(v^{l+1}_1,v^{l+1}_j) >1$. Then, it is clear that we can draw this set of paths as a subgraph of the $2$-dimensional grid in such a way that:
\begin{itemize}
\item $v^1_i$ is located at position $(i-1,0)$ in the grid, for $i\in \{1,2,\ldots,k_1\}$.
\item The vertices of $P^2$ coincide with vertices of row $1$ in the grid. Vertex $v_1^2$ is right on top of vertex $v^1_{k_1}$, that is, at position $(k_1-1,1)$ and the remaining vertices of $P^2$ are to the left of $v_1^2$.
\item The vertices of $P^3$ coincide with vertices of row $2$ in the grid. Vertex $v_1^3$ is right on top of vertex $v^2_{k_2}$, that is, at position $(k_1-k_2,2)$ and the remaining vertices of $P^3$ are to the right of $v_1^3$.
\item The embedding of the remaining paths follow the same pattern.
\end{itemize}
Thus, we can introduce the edges $\{v^i_{k_i}v^{i+1}_1\}_{i=1}^{n-1}$ in order to obtain a normalized embedding of $P_{m}$, $m=\sum_{i=1}^nk_i$, in the grid. By replacing this set of vertical edges by the edges of $T$ that do not belong to the linear forest, joining consecutive paths in the linear forest, we obtain $T$. Hence, $T$ comes from a normalized embedding of a path in the $2$-dimensional grid by a sequence of elementary transformations. Therefore, $T$ is a path-like tree.

Let us prove the converse. Assume that $T$ is a path-like tree. By Lemma \ref{lemma: every plt from norma embb}, $T$ can be obtained from a normalized embedding of a path by a sequence of elementary trasformations. Furthermore, this embedding can be chosen in such a way that if  $i$ and $i+1$ are two consecutive rows in the normalized embedding, the vertical edge that joins vertices of these rows has been moved to obtain $T$.
Then, we can take the linear forest obtained from the rows of the normalized embedding and embed this linear forest in a horizontal line. We do this in such a way that if one horizontal subpath is immediately below to another horizontal path in the normalized embedding then the path below is immediate to the left of the path above when embedding the linear forest in the horizontal line. Thus, the vertical edges of the embedding will join vertices that belong to consecutive paths in a linear configuration, having the properties described above.
\end{proof}

From now on, we will call a path-like tree drawn in the way described above, to be a path-like tree drawn in {\it linear configuration}.
Notice that, a path-like tree may admit different possible linear configurations. That is, two different linear configurations may produce isomorphic path-like trees.

\subsection{A tree of type $H$}
A tree of type $H$ is a tree such that its degree sequence has the following form: $1,1,1,1,2,2,\ldots, 2,3,3$ and such that the two vertices of degree $3$ are at distance $r\ge 1$ in $H$. In other words, a tree of type $H$ is a tree that consists of two paths joined by a path of length $r$. More precisely, we say that $T$ is a {\it tree of type $H$ with parameters $(s,t;r; a,b)$}, $1< a< s$, $1< b< t$, if $V(T)=\{u_1,\ldots,u_s,v_1,\ldots,v_t,w_1,\ldots,w_{r+1}\}$ such that  $u_a=w_1$, $v_b=w_{r+1}$ and $E(T)=\{u_1u_2,u_2u_3,\ldots,u_{s-1}u_s\}\cup\{v_1v_2,v_2v_3,\ldots,v_{t-1}v_t\}\cup \{w_1w_2,w_2w_3,\ldots,w_{r}w_{r+1}\}.$

Our immediate goal is to characterize path-like trees of type $H$.

Let $T$ be a tree drawn in linear configuration. Let $\omega_1$ and $\omega_2$ be two consecutive vertices of $T$ of degree $3$ (with no vertices of degree 4 between them), where $\omega_2$ is on the right side of $\omega_1$. Assume that $\omega_i\in P^l_{k_l}$.  We use $\omega_i^+$ to indicate that $\omega_i$ is adjacent to a vertex of $P^{l+1}_{k_{l+1}}$ and $\omega_i^-$ to indicate that $\omega_i$ is adjacent to a vertex of $P^{l-1}_{k_{l-1}}$. Then, according to the possible configurations $(\omega_1^{\epsilon_1}, \omega_2^{\epsilon_2})$, $\epsilon_i\in\{+,-\}$, for $i=1,2$, we have six different models of induced subpaths in $T$. 
\begin{enumerate}
  \item[(a)] We find $(\omega_1^-, \omega_2^+)$ and $\omega_1,\omega_2\in P^{l}_{k_{l}}$. The linear configuration contains three consecutive paths $P^{l-1}_{k_{l-1}},P^{l}_{k_l}, $ and $P^{l+1}_{k_{l+1}}$ and two edges of the form $v^{l-1}_\alpha v^{l}_i$ and $v^{l}_{i+r}v^{l+1}_{\beta}$, for some $1<i<k_{l}-1$ and $1\le r < k_{l}-i$. The path induced by $v^{l}_i,v^{l}_{i+1},\ldots, v^{l}_{i+r}$ is called a {\it parallel path} and it is denoted by $\pi$.
  
  \item[(b)] We find $(\omega_1^+, \omega_2^-)$ and $\omega_1,\omega_2\in P^{l}_{k_{l}}$. The linear configuration contains three consecutive paths $P^{l-1}_{k_{l-1}}, P^{l}_{k_{l}}$ and $P^{l+1}_{k_{l+1}}$ and two edges of the form $v^{l-1}_\alpha v^{l}_i$ and $v^{l}_{i-r}v^{l+1}_{\beta}$,  for some $2<i<k_{l}$ and $1\le r <i-1$. The path induced by $v^{l}_i,v^{l}_{i-1},\ldots, v^{l}_{i-r}$ is called a {\it crossed path} and it is denoted by $\gamma$.
  
   \item[(c)] We find $(\omega_1^+, \omega_2^-)$, $\omega_1\in P^l_{k_l}$ and $\omega_2\in P^{l+m}_{k_{l+m}}$, for some $m\ge 1$. The linear configuration contains $m+1\ge 2$ consecutive paths $P^{l}_{k_l}, P^{l+1}_{k_{l+1}}, \ldots, P^{l+m}_{k_{l+m}}$ and edges of the form $v^{l}_iv^{l+1}_{k_{{l+1}}}, v^{l+1}_{1}v^{l+2}_{k_{{l+2}}},\ldots, v^{l+m-2}_{1}v^{l+m-1}_{k_{{l+m-1}}}, v^{l+m-1}_{1}v^{l+m}_{j}$, for some $1<i<k_l$ and some $1<j<k_{l+m}$. The path connecting  $v^{l}_i$ to $v^{l+m}_{j}$ is called a {\it bridge path} and it is denoted by $\beta$.
\item[(d)] We find $(\omega_1^-, \omega_2^+)$, $\omega_1\in P^l_{k_l}$ and $\omega_2\in P^{l+m}_{k_{l+m}}$, for some $m\ge 1$. The linear configuration contains $m+3\ge 4$ consecutive paths $P^{l-1}_{k_{l-1}}, P^{l}_{k_{l}}, \ldots, P^{l+m+1}_{k_{l+m+1}}$ and edges of the form $v^{l-1}_{\alpha}v^{l}_i, v^{l}_{1}v^{l+1}_{k_{{l+1}}},\ldots, v^{l+m-1}_{1}v^{l+m}_{k_{{l+m}}}, v^{l+m}_{j}v^{l+m+1}_{\beta}$, for some $1<i<k_l$ and some $1<j<k_{l+m}$. The path connecting  $v^{l}_i$ to $v^{l+m}_{j}$ is called an {\it indirect path} and it is denoted by $\mu$.

\item[(e)] We find $(\omega_1^-, \omega_2^-)$, $\omega_1\in P^l_{k_l}$ and $\omega_2\in P^{l+m}_{k_{l+m}}$, for some $m\ge 1$. The linear configuration contains $m+2\ge 3$ consecutive paths $P^{l-1}_{k_{l-1}}, P^{l}_{k_{l}}, \ldots, P^{l+m}_{k_{l+m}}$ and edges of the form $v^{l-1}_{\alpha}v^{l}_i, v^{l}_{1}v^{l+1}_{k_{{l+1}}},\ldots, v^{l+m-2}_{1}v^{l+m-1}_{k_{{l+m-1}}}, v^{l+m-1}_{1}v^{l+m}_{j}$, for some $1<i<k_l$ and some $1<j<k_{l+m}$. The path connecting  $v^{l}_i$ to $v^{l+m}_{j}$ is called a {\it semi-indirect path} and it is denoted by $\sigma^-$.

\item[(f)] We find $(\omega_1^+, \omega_2^+)$, $\omega_1\in P^l_{k_l}$ and $\omega_2\in P^{l+m}_{k_{l+m}}$, for some $m\ge 1$.The linear configuration contains $m+2\ge 3$ consecutive paths $P^{l}_{k_{l}},P^{l+1}_{k_{l-1}},  \ldots, P^{l+m}_{k_{l+m}}, P^{l+m+1}_{k_{l+m+1}}$ and edges of the form $v^{l}_{i}v^{l+1}_{k_{{l+1}}},v^{l+1}_{1}v^{l+2}_{k_{{l+2}}}\ldots, v^{l+m-1}_{1}v^{l+m}_{k_{{l+2}}},v^{l+m-1}_{1}v^{l+m}_{k_{{l+m-1}}}, v^{l+m}_{j}v^{l+m+1}_{\beta}$, for some $1<i<k_l$ and some $1<j<k_{l+m}$. The path connecting  $v^{l}_i$ to $v^{l+m}_{j}$ is also called a {\it semi-indirect path} and it is denoted by $\sigma^+$.

\end{enumerate}

Then, we have the following proposition.

\begin{proposition}\label{theo: charac plt type H with pi or gamma}
Let $T$ be a tree of type $H$ with parameters $(s,t; r;a,b)$.
\begin{enumerate}
  \item[(i)] $T$ is a path-like tree with a parallel path $\pi$ if and only if $a$ is a divisor of $s$ 
and $b$ is a divisor of $t$, where $a\le s-a$ and $b\le t-b$.
  \item[(ii)] $T$ is a path-like tree with a crossed path $\gamma$ if and only if
$(a+r)$ is a divisor of $(t-b)$ and $(b+r)$ is a divisor of $(s-a)$, where $a+r\le t-b$ and $b+r\le s-a$.
\end{enumerate}
\end{proposition}

\begin{proof}
(i) Assume that $T$ is a path-like tree of type $H$ with parameters $(s,t;r;a,b)$, which contains a parallel path, and let
$P^1_{k_1},P^2_{k_2}, \ldots, P^n_{k_n}$ be the subpaths of a linear configuration of the path-like tree. Assume that $v^l_a$ and $v^l_{a+r}$ are the only two vertices of degree $3$, for $2\le l\le n-1$ and $1<a<a+r<k_l$. It follows that $k_{l-1}=a$, since the edge that joins $v^l_a$ with the vertex of $P^{l-1}_{k_{l-1}}$ has to be incident with a terminal vertex of $P^{l-1}_{k_{l-1}}$, for if not, some vertex of $P^{l-1}_{k_{l-1}}$ would have degree $3$ in $T$, and this is impossible since all vertices of degree $3$ in $T$ are $v^l_a$ and $v^l_{a+r}$. A similar reasoning shows that $k_{1}=k_{2}=\ldots =k_{l-1}=a$. In a similar way, we have that $k_{l+1}=k_{l+2}=\ldots =k_{n}=k_l-(a+r)+1$. Let us prove the converse. Assume that $T$ is a tree of type $H$ with parameters $(s,t;r;a,b)$ such that $a$ is a divisor of $(s-a)$ 
and $b$ is a divisor of $(t-b)$. Then, there exist integers $x$ and $y$ with $ax=(s-a)$ and $by=(t-b)$. Let $P^1_{k_1}\cong P^2_{k_2}\cong\ldots \cong P^x_{k_x}\cong P_a$, $P^{x+1}_{k_{x+1}}\cong P_{a+r-1+b}$ and $P^{x+2}_{k_{x+1}}\cong P^{x+3}_{k_{x+3}}\cong\ldots \cong P^{x+y+1}_{k_{x+y+1}}\cong P_b$. Then we embed the subpaths in a horizontal line as described in the definition of the linear configuration and we join the subpaths $P^i_{k_i}$ and $P^{i+1}_{k_{i+1}}$, $i=1,2,\ldots, x+y$, with an edge in such a way that we only produce two vertices of degree $3$, at distance $r$, both such vertices belonging to the subpath $P^{x+1}_{k_{x+1}}$, and never joining the last vertex of $P^i_{k_i}$ with the first vertex of $P^{i+1}_{k_{i+1}}$, $i=1,2,\ldots, x+y$. It is easy to see that this produces a linear configuration of the tree $T$, as described in the beginning of the section. Therefore, $T$ is a path-like tree with a parallel path.

(ii)  Assume that $T$ is a path-like tree of type $H$ with parameters $(s,t;r;a,b)$, which contains a crossed path, and let $P^1_{k_1},P^2_{k_2}, \ldots, P^n_{k_n}$ be the subpaths of a linear configuration of the path-like tree. Assume that $v^l_a$ and $v^l_{a+r}$ are the only two vertices of degree $3$, for $2\le l\le n-1$ and $1<a<a+r<k_l$. It follows that $k_{l-1}=a+r$, since the edge that joins vertex $v^l_{a+r}$ with the vertex of $P_{k_{l-1}}^{l-1}$ has to be incident with a terminal vertex of $P_{k_{l-1}}^{l-1}$. For if not, some vertex of $P_{k_{l-1}}^{l-1}$ would have degree $3$ in $T$, and this is impossible since all the vetices of degree $3$ in $T$ are $v^l_a$ and $v^l_{a+r}$. A similar reasoning shows that $k_{1}=k_{2}=\ldots =k_{l-1}=a+r$. In a similar way, we have that $k_{l+1}=k_{l+2}=\ldots =k_{n}=k_l-a+1$. Let us prove the converse. Assume that $T$ is a tree of type $H$ with parameters $(s,t;r;a,b)$ such that $(a+r)$ is a divisor of $(t-b)$ and $(b+r)$ is a divisor of $(s-a)$, where $a+r\le t-b$ and $b+r\le s-a$. Then, there exist integers $x$ and $y$ with $(a+r)x=t-b$ and $(b+r)y=s-a$. Let $P^1_{k_1}\cong P^2_{k_2}\cong\ldots P^x_{k_x}\cong P_{a+r}$, $P^{x+1}_{k_{x+1}}\cong P_{a+r-1+b}$ and $P^{x+2}_{k_{x+2}}\cong P^{x+3}_{k_{x+3}}\cong\ldots \cong P^{x+y+1}_{k_{x+y+1}}\cong P_{b+r}$. Then, we embed the subpaths in a horizontal line as described in the definition of linear configuration, and we join the subpaths $P^i_{k_i}$ and $P^{i+1}_{k_{i+1}}$, $i=1,2,\ldots, x+y$, with an edge in such a way that we only produce two vertices of degree $3$, at distance $r$, both such vertices belonging to the subpath $P^{x+1}_{k_{x+1}}$, and never joining the last vertex of $P^i_{k_i}$ with the first vertex of $P^{i+1}_{k_{i+1}}$, $i=1,2,\ldots,x+y$. It is easy to see that this produces a linear configuration of the tree $T$, as described in the beginning of the section. Therefore, $T$ is a path-like tree with a crossed path.
\end{proof}

\begin{remark}\label{remark_divisors_paths}
Let $T$ be a path-like tree of type $H$ with parameters $(s,t;r;a,b)$. By Proposition \ref{theo: charac plt type H with pi or gamma}, we can associate two integers $d_1$ and $d_2$ to $T$ which act as divisors of $s-a$ and $t-b$, subject to some constraints:
\begin{enumerate}
  \item[(i)] If $T$ is a path-like tree with a parallel path $\pi$ then $d_1=a$ (when $a\le s-a$) and $d_2=b$ (when $b\le t-b$).
 \item[(ii)] If $T$ is a path-like tree with a crossed path $\gamma$ then $d_1=b+r$ (when $b+r\le s-a$) and $d_2=a+r$ (when $a+r\le t-b$).
\end{enumerate}

We will refer to $d_1$ and $d_2$ as the horizontal divisors of $T$.
\end{remark}

The next lemma is an easy exercise.

\begin{lemma}\label{lemma: path lengths_plt_brigde}
Let $T$ be a path-like tree of type $H$ with parameters $(s,t;r;a,b)$, which contains a bridge path $\beta$ and let $P^1_{k_1},P^2_{k_2}, \ldots, P^n_{k_n}$ be the subpaths of a linear configuration of the path-like tree. Also let $P^l_{k_l}$ and $P^{l+m}_{k_{l+m}}$, $m\ge 1$, be the two paths containing the vertices of degree $3$. Then, $k_1=k_2=\ldots=k_l$ and $k_{l+m}=k_{l+m+1}=\ldots=k_n$. Moreover, if $m\ge 2$ then $k_{l+1}=k_{l+2}=\ldots=k_{l+m-1}$.

\end{lemma}

\begin{theorem}\label{theo: characterization of a H with a bridge}
Let $T$ be a tree of type $H$ with parameters $(s,t;r;a,b)$. Then, $T$ is a path-like tree with a bridge path if and only if one of the following holds:
\begin{itemize}
  \item[(i)] There exist a divisor $d_1$ of $s$, a divisor $d_2$ of $t$ and $\delta\in\mathbb{Z}$, with $2\le \delta\le \min\{d_1,d_2\}-1$ such that $d_1 -\delta=\min\{a-1,s-a\}$, $d_2-\delta=\min\{b-1, t-b\}$ and $\delta$ is a divisor of $r-1$.
  \item[(ii)] There exist $\delta\in\{a,s-a+1\}\cap \{b,t-b+1\}$ such that $\delta$ is a divisor of $r-1$.
   \item[(iii)] There exist $\delta\in \{a,s-a+1\}$ and a divisor $d_2$ of $t$, such that $d_2-\delta=\min\{b-1,t-b\}$ and $\delta$ is a divisor of $r-1$.
  \item[(iv)] There exist $\delta\in \{b,t-b+1\}$ and a divisor $d_1$ of $s$, such that $d_1-\delta=\min\{a-1,s-a\}$ and $\delta$ is a divisor of $r-1$.
\end{itemize}
\end{theorem}

\begin{proof}
Assume that $T$ is a path-like tree with a bridge path and let $P^1_{k_1},P^2_{k_2},$ \ldots, $P^n_{k_n}$ be the subpaths of a linear configuration of $T$. Assume that, the vertices of degree $3$ belong to $P^l_{k_l}$ and $P^{l+m}_{k_{l+m}}$, for $m\ge 1$.
Suppose first that $l=1$ and $l+m=n$. Then, when we introduce the path that joins a vertex of $P^1_{k_1}$ with $P^{n}_{k_n}$, property (ii) holds. Suppose now that either $l>1$ or  $l+m<n$. By Lemma \ref{lemma: path lengths_plt_brigde}, $k_1=k_2=\ldots=k_l$ and $k_{l+m}=k_{l+m+1}=\ldots=k_n$.

We distinguish three cases.

\underline{Case $l=1$.} Assume that $v^{1}_av^{2}_{\delta}\in E(T)$. Then, it is clear that $k_n -\delta=\min\{b-1,t-b\}$, $s-\delta=a-1$ and $\delta$ is a divisor of $r-1$. Depending on our choice of $s,t,a,b,s-a,t-b$, we have proven either (iii) or (iv).

\underline{Case  $1<l<l+m<n$.} Let $v^{l}_av^{l+1}_{\delta}\in E(T)$. Then, it is clear that $k_1-\delta=\min\{a-1, s-a\}$ and $k_n-\delta=\min\{b-1, t-b\}$. Thus, for $d_1=k_1$ and $d_2=k_n$ condition (i) holds.

\underline{Case $l+m=n$.} This case can be treated in a similar way to case $l=1$, with a similar result. This completes the proof of the necessity.

Let us prove the converse. That is, any tree of type $H$ with parameters $(s,t;r;a,b)$
verifying one of the conditions described in the statement of the theorem is a path-like tree. It is clear that if condition (ii) holds, then $T$ is a path-like tree. Now, since conditions (iii) and (iv) are symmetric, we will just show that if condition (iii) holds then $T$ is a path-like tree. Suppose that there exist a divisor $d_2$ of $t$, such that $d_2-\delta=b-1$ and $\delta$ is a divisor of $r-1$, when we assume $b-1\le t-b$ and $\delta\in\{a,s-a+1\}$, we will prove that $T$ is a path-like tree. Let $P^1_{s}=v^1_1v^1_2\ldots v^1_s$, $P^2_{k_2}\cong P^3_{k_3}\cong\ldots \cong P^{m}_{k_{m}}\cong P_\delta$ and
$P^{m+1}_{k_{m+1}}\cong P^{m+2}_{k_{m+2}}\cong\ldots \cong P^n_{k_n}\cong P_{d_2} =v_1v_2\ldots v_{d_2}$, where $r-1=(m-1)\delta$ and $t=(n-m)d_2$. Next, embed the paths on a horizontal line as described in the definition of the linear configuration. Now, the path that joins the vertices of degree $3$, joins a internal vertex of $P^1_{s}$ with and end vertex of $P^2_{k_2}$, and also an end vertex of $P^{m}_{k_{m}}$ with an internal vertex of $P^{m+1}_{k_{m+1}}$, when $m\ge 2$, or an internal vertex of $P^1_{s}$ with an internal vertex of $P^2_{k_2}$, when $m=1$. The rest of the edges are introduced in such a way that we do not create any more vertices of degree $3$. Notice that, there is only one possible way to do this. 
It is easy to see that the tails of the $H$-type tree are as requested, and that the tree is in fact a path-like tree.

Finally, let us prove that condition (i) implies that $T$ is a path-like tree. Assume that there exist a divisor $d_1$ of $s$, a divisor $d_2$ of $t$ and $\delta\in\mathbb{Z}$, with $2\le \delta\le \min\{d_1,d_2\}-1$ such that $d_1 -\delta=a-1$, $d_2-\delta=b-1$ and $r-1=(m-1)\delta$, where $s-a\ge a-1$ and $t-b\ge b-1$. Let $x$ and $y$ be such that $d_1 x=s$ and $d_2 y=t$, and define the subpaths:
$P^1_{k_1}, P^2_{k_2},\ldots, P^n_{k_n}$, where $n=x+y+m-1$, $k_1=k_2=\ldots=k_x=d_1$,  $k_{x+1}=k_{x+2}=\ldots=k_{x+m-1}=\delta$ and $k_{x+m}=k_{x+m+1}=\ldots=k_n=d_2$. Now, embed these subpaths in a horizontal line and introduce the edges between two consecutive paths $P^l_{k_l}$ and $P^{l+1}_{k_{l+1}}$, with $l\notin\{ x,x+m-1\}$, in such a way that no vertex of degree $3$ is created. Finally, we have to create the two vertices of degree $3$ in $P^x_{k_x}$ and $P^{x+m}_{k_{x+m}}$. We do this by joining $v^x_a$ with $v^{x+1}_{\delta}$, (and joining $v^{x+m-1}_{1}$ with $v^{x+m}_{\delta}$, when $m\ge 2$). Clearly, $d(v^x_a,v^x_{k_x})=\delta-1=d(v^{x+m}_1,v^{x+m}_{\delta})$, therefore, $T$ is a path-like tree.

\end{proof}

\begin{remark}\label{remark_divisors_brigde}
Let $T$ be a path-like tree of type $H$ with parameters $(s,t;r;a,b)$ and a bridge path $\beta$. By the previous theorem, we can associate to $T$, a divisor $d_1$ of $s$, a divisor $d_2$ of $t$ and a divisor $\delta$ of $r-1$, such that, according to the four cases of Theorem \ref{theo: characterization of a H with a bridge}:

\begin{itemize}
  \item[(i)] $2\le \delta\le \min\{d_1,d_2\}-1$, $d_1 -\delta=\min\{a-1,s-a\}$ and $d_2-\delta=\min\{b-1, t-b\}$.
  \item[(ii)] $d_1=s$, $d_2=t$ and $\delta\in \{a,s-a+1\}\cap \{b,t-b+1\}$.
   \item[(iii)] $d_1=s$, $d_2-\delta=\min\{b-1,t-b\}$ and $\delta\in \{a,s-a+1\}$.
  \item[(iv)] $d_1-\delta=\min\{a-1,s-a\}$, $d_2=t$ and $\delta\in \{b,t-b+1\}$.
\end{itemize}

We will say that $d_1$ and $d_2$ are horizontal divisors of $T$ and $\delta$ is a vertical divisor.
\end{remark}

\begin{theorem}\label{theo: characterization of a H with a indirect path}
Let $T$ be a tree of type $H$ with parameters $(s,t;r;a,b)$. Then, $T$ is a path-like tree with an indirect path if and only if 
 there exist $\phi_1\in \{a, s-a+1\}$, $\phi_2\in \{b,t-b+1\}$, a divisor $d_1$ of $s-\phi_1$ and a divisor $d_2$ of $t-\phi_2$ such that $d_1+\phi_1-1=d_2+\phi_2-1$ is a divisor of $r+1-d_1-d_2\ge 0$.
\end{theorem}  
\begin{proof}
Let $T$ be a path-like tree which contains an indirect path. Assume that the linear configuration of the path-like tree, with subpaths $P^1_{k_1},P^2_{k_2}, \ldots, P^n_{k_n}$, contains an indirect path $\mu$ and let $v_a^{l}$ and $v_b^{l+m}$ be the only vertices of degree $3$, for $2\le l<l+m< n$. It follows that $k_{l-1}=a$, since the edge that joins $v^l_a$ with the vertex of $P^{l-1}_{k_{l-1}}$ has to be incident with a terminal vertex of $P^{l-1}_{k_{l-1}}$, for if not, some vertex of $P^{l-1}_{k_{l-1}}$ would have degree $3$ in $T$, and this is impossible since all vertices of degree $3$ in $T$ are $v^l_a$ and $v_b^{l+m}$. A similar reasoning shows that $k_{1}=k_{2}=\ldots =k_{l-1}=a$. In a similar way, we have that $k_{l+m+1}=k_{l+m+2}=\ldots =k_{n}=k_{l+m}-b+1$, and $k_l=k_{l+1}=k_{l+2}=\ldots =k_{l+m}$. Depending on our choice of $a,b,s-a,t-b$, we have proven one of the four possibilities of the theorem.

Let us prove the converse. Suppose that there exist $\phi_1\in\{a,s-a+1\}$, $\phi_2\in\{b,t-b+1\}$, a divisor $d_1$ of $s-\phi_1$ and a divisor $d_2$ of $t-\phi_2$ such that, $d_1+\phi_1-1=d_2+\phi_2-1$ is a divisor of $r+1-d_1-d_2$.  Let $P^1_{k_1}\cong P^2_{k_2}\cong\ldots \cong P^{x}_{k_{x}}\cong P_{d_1}$,  $P^{x+1}_{k_{x+1}}\cong P^{x+2}_{k_{x+2}}\cong\ldots \cong P^{x+m+1}_{k_{x+m+1}}\cong P_{d_1+\phi_1-1}$ and $P^{x+m+2}_{k_{x+m+2}}\cong P^{x+m+3}_{k_{x+m+3}}\cong\ldots \cong P^{x+m+y+1}_{k_{x+m+y+1}}\cong P_{d_2}$, where $x,y$ and $m-1$ are integers such that $s-\phi_1=xd_1$, $t-\phi_2=yd_2$ and $r+1-d_1-d_2=(m-1)(d_1+\delta-1)$. Next, embed the paths on a horizontal line as described in the definition of the linear configuration. Now, the vertices of degree $3$ are defined when we join a internal vertex of $P^{x+1}_{k_{x+1}}$, namely, $v^{x+1}_{k_{x+1}-\phi_1+1}$ with and end vertex of $P^x_{k_x}$, and also an end vertex of $P^{x+m+2}_{k_{x+m+2}}$ with an internal vertex of $P^{m+x+1}_{k_{m+x+1}}$, namely $v^{m+x+1}_{\phi_2}$. The rest of the edges are introduced in such a way that we do not create any more vertices of degree $3$. Notice that, there is only one possible way to do this. It is easy to see that the tails of the $H$-type tree are as requested, and that the tree is in fact a path-like tree.

\end{proof}
\begin{remark}\label{remark_divisors_indirect}
Let $T$ be a path-like tree of type $H$ with parameters $(s,t;r;a,b)$ and an indirect path $\mu$. By the previous theorem, we can associate to $T$, a divisor $d_1$ of $s-\phi_1$, a divisor $d_2$ of $t-\phi_2$ and a divisor $\delta=d_1+\phi_1-1=d_2+\phi_2-1$ of $r+1-d_1-d_2$, where $\phi_1\in\{a,s-a+1\}$ and $\phi_2\in\{b,t-b+1\}$.

We will say that $d_1$ and $d_2$ are horizontal divisors of $T$ and $\delta$ is a vertical divisor.
\end{remark}

\begin{theorem}\label{theo: characterization of a H with a semi-indirect path}
Let $T$ be a tree of type $H$ with parameters $(s,t;r;a,b)$. Then, $T$ is a path-like tree which contains a semi-indirect path if and only if one of the following holds:
\begin{itemize}
  \item[(i)] There exist $\phi_1\in\{a,s-a+1\}$, a divisor $d_1$ of $s-\phi_1$ and a divisor $d_2$ of $t$ such that, $d_1+\phi_1-1$ is a divisor of $r-d_1$, and $d_2-(d_1+\phi_1-1)\in \{b-1,t-b\}$.
  \item[(ii)] There exist $\phi_2\in\{b,t-b+1\}$, a divisor $d_2$ of $t-\phi_2$ and a divisor $d_1$ of $s$ such that, $d_2+\phi_2-1$ is a divisor of $r-d_2$, and $d_1-(d_2+\phi_2-1)\in \{a-1,s-a\}$.
\end{itemize}
\end{theorem}  

\begin{proof}
Let $T$ be a path-like tree. Assume that the linear configuration of the path-like tree, with subpaths $P^1_{k_1},P^2_{k_2}, \ldots, P^n_{k_n}$, contains a semi-indirect path $\sigma^-$ and let $v_a^{l}$ and $v_\delta^{l+m}$ be the only vertices of degree $3$, for $2\le l<l+m\le n$. It follows that $k_{l-1}=a$, since the path that joins $v^l_a$ with the vertex of $P^{l-1}_{k_{l-1}}$ has to be incident with a terminal vertex of $P^{l-1}_{k_{l-1}}$, for if not, some vertex of $P^{l-1}_{k_{l-1}}$ would have degree $3$ in $T$, and this is impossible since all vertices of degree $3$ in $T$ are $v^l_a$ and $v^{l+m}_\delta$. A similar reasoning shows that $k_{1}=k_{2}=\ldots =k_{l-1}=a$. In a similar way, we have that $k_{l+m}=k_{l+m+1}=\ldots =k_{n}$, and $k_l=k_{l+1}=k_{l+2}=\ldots =k_{l+m-1}$. Depending on our choice of $s,t,a,b,s-a,t-b$, we have proven either (i) or (ii).

Let us prove the converse. That is, any tree of type $H$ with parameters $(s,t;r;a,b)$
verifying one of the conditions described in the statement of the theorem is a path-like tree. It is clear that conditions (i) and (ii) are symmetric. Hence, we will just show that if condition (i) holds then $T$ is a path-like tree.
Suppose that there exist $\phi_1\in\{a,s-a+1\}$, a divisor $d_1$ of $s-\phi_1$ and a divisor $d_2$ of $t$ such that, $d_1+\phi_1-1$ is a divisor of $r-d_1$,  and either $d_2-b+1=d_1+\phi_1-1$ or $d_2-t+b=d_1+\phi_1-1$. Suppose first that $d_2-b+1=d_1+\delta-1$. Let $P^1_{k_1}\cong P^2_{k_2}\cong\ldots \cong P^{x}_{k_{x}}\cong P_{d_1}$,  $P^{x+1}_{k_{x+1}}\cong P^{x+2}_{k_{x+2}}\cong\ldots \cong P^{x+m}_{k_{x+m}}\cong P_{d_1+\delta-1}$ and $P^{x+m+1}_{k_{x+m+1}}\cong P^{x+m+2}_{k_{x+m+2}}\cong\ldots \cong P^{x+m+y}_{k_{x+m+y}}\cong P_{d_2}$, where $x,y$ and $m-1$ are integers such that $s-\phi_1=xd_1$, $t=yd_2$ and $r-d_1=(m-1)(d_1+\phi_1-1)$. Next, embed the paths on a horizontal line as described in the definition of the linear configuration. Now, the path that joins the vertices of degree $3$, joins a internal vertex of $P^{x+1}_{k_{x+1}}$, namely, $v^{x+1}_{k_{x+1}-\phi_1+1}$ with and end vertex of $P^x_{k_x}$, and also an end vertex of $P^{x+m}_{k_{x+m}}$ with an internal vertex of $P^{m+x+1}_{k_{m+x+1}}$, namely $v^{m+x+1}_{k_{m+x+1}-b+1}$. The rest of the edges are introduced in such a way that we do not create any more vertices of degree $3$. Notice that, there is only one possible way to do this. 
It is easy to see that the tails of the $H$-type tree are as requested, and that the tree is in fact a path-like tree. A similar construction works when we assume that  $d_2-t+b=d_1+\delta-1$.
\end{proof}

\begin{remark}\label{remark_divisors_semi-indirect}
Let $T$ be a path-like tree of type $H$ with parameters $(s,t;r;a,b)$ and a semi-indirect path $\sigma$. According to the two cases of the previous theorem, we can associate to $T$: 
\begin{enumerate}
\item[(i)] Either, a divisor $d_1$ of $s-\phi_1$, a divisor $d_2$ of $t$ and a divisor $\delta= d_1+\phi_1-1$ of $r-d_1$, where $\phi_1\in\{a,s-a+1\}$ and $d_2-\delta\in\{b-1,t-b\}$.
\item[(ii)] Or, a divisor $d_1$ of $s$, a divisor $d_2$ of $t-\phi_2$ and a divisor $\delta=d_2+\phi_2-1$ of $r-d_2$, where $\phi_2\in\{b,t-b+1\}$ and $d_1-\delta\in\{a-1,s-a\}$.
\end{enumerate}

We will say that $d_1$ and $d_2$ are horizontal divisors of $T$ and $\delta$ is a vertical divisor.
\end{remark}
\subsection{Trees of type cutted $H$}\label{subsection: cutted $H$}
A tree of type cutted $H$, simply denoted as $cH$ is a tree such that
the degree sequence of the tree has the following form: $1,1,1,2,2,\ldots, 2,3$. In other words, a tree of type $cH$ is a tree, different from a path, obtained by joining two vertices of two different paths, one of them being a leave and the other one not. More precisely, we say that $T$ is a {\it tree of type $cH$ with parameters $(s,t;a)$} if $V(T)=\{u_1,\ldots,u_s,v_1,\ldots,v_t\}$ such that $1< a< s$ and $E(T)=\{u_av_1\}\cup \{u_1u_2,u_2u_3,\ldots,u_{s-1}u_s\}\\ \cup\{v_1v_2,v_2v_3,\ldots,v_{t-1}v_t\}.$

Now, we will characterize path-like trees of type $cH$.

Let $T$ be a path-like tree of type $cH$. Then, a linear configuration contains two consecutive paths $P^{l}_{k_l}$ and $ P^{l+1}_{k_{l+1}}$ and one edge of the form either $v^{l}_iv^{l+1}_{k_{l+1}}$ or  $v^{l}_1v^{l+1}_{j}$, for $1<i<k_l$, $1<j<k_{l+1}$.

Then, we have the following theorem.

\begin{theorem}
Let $T$ be a tree of type $cH$ with parameters $(s,t;a)$. Then, $T$ is a path-like tree if and only if one of the following holds:
\begin{itemize}
  \item[(i)] There exist a divisor $\alpha$ of $s$, a divisor $\beta$ of $t$ such that $\alpha -\beta=a-1$, where $s-a\ge a-1$.
  \item[(ii)] Either $s-a+1=t$ or $a=t$.
   \item[(iii)] Either $a$ or  $s-a+1$ is a divisor of $t$.
\end{itemize}
\end{theorem}

\begin{proof}
Assume that $T$ is a path-like tree of type $cH$ with parameters $(s,t;a)$ and let $P^1_{k_1},P^2_{k_2}, \ldots, P^n_{k_n}$ be the subpaths of a linear configuration of the path-like tree.
Suppose first that $n=2$. Then, when we introduce the edge that joins a vertex of $P^1_{k_1}$ with $P^2_{k_2}$, property (ii) holds. Suppose now that $n\ge 3$ and let $l\in \{1,2,\ldots,n-1\}$ such that $v^{l}_av^{l+1}_{k_{l+1}}\in E(T)$, with $v^{l}_a\in P^l_{k_l}$ and $v^{l+1}_{k_{l+1}}\in P^{l+1}_{k_{l+1}}$. Similar to what happens in Lemma \ref{lemma: path lengths_plt_brigde}, it is easy to check that $k_1=k_2=\ldots=k_l$ and $k_{l+1}=k_{l+2}=\ldots=k_n$. Clearly, $k_1> k_n$, otherwise $T$ is a path, which is a contradiction, since it contains a vertex of degree $3$. We distinguish two cases.

\underline{Case $l=1$.} Since $v^{1}_av^{2}_{k_2}\in E(T)$, we obtain that, $s-a+1$ is a divisor of $t$. By replacing $s-a+1$ by $a$, we obtain (iii)

\underline{Case  $1<l<n$.} Let $v^{l}_av^{l+1}_{k_{l+1}}\in E(T)$. Then, it is clear that $k_1-{k_{l+1}}=a-1$. Thus, for $\alpha=k_1$, $\beta=k_n$ condition (i) holds, when $s-a\ge a-1$.

Let us prove the converse. That is, any tree of type $cH$ with parameters $(s,t;a)$ verifying one of the conditions described above is a path-like tree. It is clear that if condition (ii) holds, $T$ is a path-like tree. Now, we will show that if condition (iii) holds then $T$ is a path-like tree. Suppose that $s-a+1$ is a divisor of $t$ and let $x(s-a+1)=t$.

Let $P^1_{s}=v^1_1v^1_2\ldots v^1_s$ and
$P^2_{k_2}\cong P^3_{k_3}\cong\ldots \cong P^{1+x}_{k_{1+x}}\cong P_{s-a+1} =v_1v_2\ldots v_{s-a+1}$, where $t=(s-a+1)x$. Next, embed the paths on a horizontal line as described in the definition of the linear configuration. Now, the edge that joins the vertices of degree $3$, joins a vertex of $P^1_{k_1}$ with a vertex of $P^2_{k_2}$. The rest of the edges are introduced in such a way that we do not create any more vertices of degree $3$. Notice that, there is only one possible way to do this. Thus, now let us go back, in order to explain how the edge that joins a vertex of $P^1_{k_1}$ with a vertex of $P^2_{k_2}$ is introduced. It is enough to join vertex $v^1_a$ to vertex $v^2_{k_2}$. It is easy to see that the $cH$-type tree is in fact a path-like tree.

Finally, let us prove that condition (i) implies that $T$ is a path-like tree. Assume that there exist a divisor $\alpha$ of $s$, a divisor $\beta$ of $t$ such that $\alpha -\beta=a-1$, where $s-a\ge a-1$. Let $x$ and $y$ be such that $\alpha x=s$ and $\beta y=t$, and define the subpaths:
$P^1_{k_1}, P^2_{k_2},\ldots, P^{x+y}_{k_{x+y}}$, $k_1=k_2=\ldots=k_x=\alpha$ and $k_{x+1}=k_{x+2}=\ldots=k_{x+y}=\beta$. Now, embed these subpaths in a horizontal line and introduce the edges between two consecutive paths $P^l_{k_l}$ and $P^{l+1}_{k_{l+1}}$, with $l\neq x$, in such a way that no vertices of degree $3$ are created. Finally, we have to create the vertex of degree $3$ in $P^x_{k_x}$. We do this by joining $v^x_a$ with $v^{x+1}_{k_{x+1}}$. Clearly, $d(v^x_a,v^x_{k_x})=d(v^{x+1}_1,v^{x+1}_{k_{x+1}})$, therefore, $T$ is a path-like tree.
\end{proof}

\section{Path-like trees of maximum degree $3$}

Let $H_n=(h_1,h_2,\ldots,h_n)$, $R_n=(r_1,r_2,\ldots, r_{n})$, $A_n=(a_1,a_2,\ldots,a_n)$ and $B_n=(b_1,b_2,\ldots,b_n)$. We say that $T$ is a tree in $\mathcal{H}_n$ with parameters $(H_n;R_{n-1};A_{n-1},B_{n-1})$ if $T$ is a tree formed by a set of horizontal paths $\overline{P}^1_{h_1}, \overline{P}^2_{h_2}, \ldots, \overline{P}^n_{h_n}$ joined by vertical paths of lengths $r_1,r_2,\ldots, r_{n-1}$ such that every vertical path $P^i_{r_i+1}$ defines two subpaths of order $a_i$ and $b_i$ in $\overline{P}^i_{h_i}$ and $\overline{P}^{i+1}_{h_{i+1}}$, respectively as shows Fig. \ref{Fig: A tree of the form}.

\begin{figure}[h]
  \includegraphics[width=290pt]{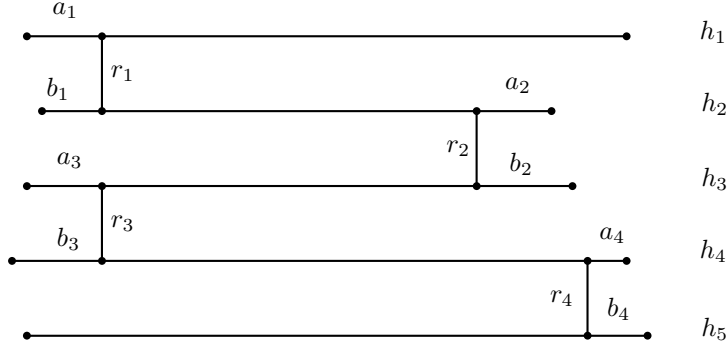}\\
  \caption{A tree in $\mathcal{H}_5$.}\label{Fig: A tree of the form}
\end{figure}

Let $P^1_{k_1}, P^2_{k_2}, \ldots,P^m_{k_m}$ be the paths defined by a linear configuration of a path-like tree.
Then every path $P_{k_i}^i$ contains at most, two vertices of degree $3$. We say that $T$ is a path-like tree of {\it type 1},
if it contains an even number of vertices of degree $3$.

The next lemma is an easy exercise. 

\begin{lemma}\label{lemma: plt of type 1_sequences}
Let $T$ be a path-like tree of type $1$ with maximum degree three. Then $T\in \mathcal{H}_{n}$, where $2(n-1)$ is the number of vertices of degree $3$. 
Moreover, every linear configuration of $T$ defines a sequence $(\alpha_1,\alpha_2,\ldots,\alpha_{n-1})$, where $\alpha_i\in\{\pi, \gamma,\beta,\mu,\sigma^+,\sigma^-\}$.
\end{lemma}
\begin{proof}
Let $P^1_{k_1}, P^2_{k_2}, \ldots,P^m_{k_m}$ be the paths defined by a linear configuration of $T$ and let $\{u_1,v_1,u_2,v_2,\ldots,u_{n-1},v_{n-1}\}$ be the set of ordered vertices  of degree $3$ (according to the order in a path that is incident to all vertices of degree $3$, that follows the linear configuration from left to right side). Let $w\in P^l_{k_l}$, we use  $w^+$ to indicate that $w$ is adjacent to a vertex in $P^{l+1}_{k_{l+1}}$ and $w^-$ to indicate that $w$ is adjacent to a vertex in $P^{l-1}_{k_{l-1}}$.
We  define a set of horizontal paths $\overline{P}^1_{h_1}, \overline{P}^2_{h_2}, \ldots, \overline{P}^n_{h_n}$ such that,  $u_i\in \overline{P}^i_{h_i}$ and $v_i\in \overline{P}^{i+1}_{h_{i+1}}$. Thus, we define $n-1$ vertical paths with end vertices $u_i$ and $v_i$, for   $i=1,2,\ldots,n-1$, each of them related to:
\begin{itemize}
\item  either a parallel path $\pi$  or a crossed path $\gamma$ when $u_i, v_i\in  P^l_{k_l}$;
\item or a bridge path $\beta$ when $u_i^+\in  P^l_{k_l}$ and $v_i^- \in  P^{l+m}_{k_{l+m}}$ and $m\ge 1$;
\item or an indirect path $\mu$ when $u_i^-\in  P^l_{k_l}$ and $v_i^+\in  P^{l+m}_{k_{l+m}}$ and $m\ge 1$;
\item or a semi-indirect path $\sigma^-$ when $u_i^-\in  P^l_{k_l}$ and $v_i^-\in  P^{l+m}_{k_{l+m}}$ and $m\ge 1$;
\item or a semi-indirect path $\sigma^+$ when $u_i^+\in  P^l_{k_l}$ and $v_i^+\in  P^{l+m}_{k_{l+m}}$ and $m\ge 1$.
\end{itemize}
\end{proof}

By the characterizations of path-like trees of type $H$ we will introduce a set of divisors $d_i$, $i=1,2,\ldots,n$ each of them related to $\overline{P}^i_{h_i}$ and also $\delta^j$, $j=1,2,\ldots,n-1$, related to the vertical paths.  Fig. \ref{Fig: A path-like tree}
and \ref{Fig: A path-like tree h_4}
show their construction in some of the cases. Theorem \ref{theo: a path-like-tree. Then...} shows how these divisors can be obtained.
\begin{figure}[h]

\includegraphics[width=391pt]{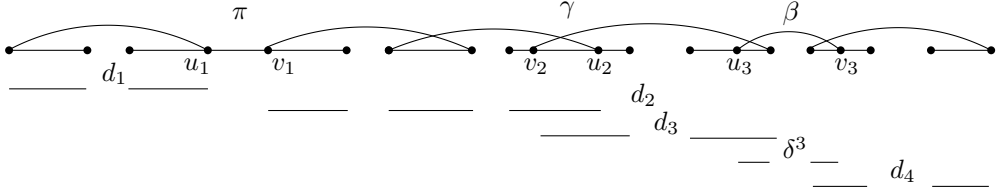}\\

  \caption{A path-like tree $T$.}\label{Fig: A path-like tree}
\end{figure}

\begin{figure}[h]
\begin{center}
  \includegraphics[width=133pt]{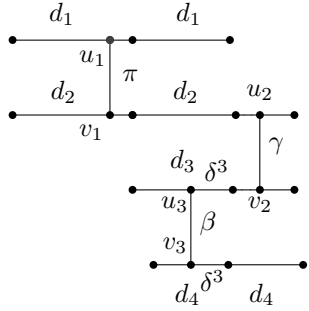}\\
  \caption{A path-like tree $T$ in  $\mathcal{H}_4$.}\label
{Fig: A path-like tree h_4}
\end{center}
\end{figure}

To simplify the notation, if we deal with a tree in $\mathcal{H}_n$ with parameters $(H_n;R_{n-1};A_{n-1},\\ B_{n-1})$ and $a_n$ appears in a formula, then we will assume that $a_n=0$.  Also, if we refer to a vertical divisor $\delta$ of a path of type $\pi$ or $\gamma$, we will assume that $\delta=1$, since no vertical divisor has been introduced in this case.

\begin{theorem}\label{theo: a path-like-tree. Then...}
Let $T$ be a path-like tree of type 1 with maximum degree three and set of ordered pairs of vertices (according to a linear configuration of the path-like tree) of degree $3$, $\{(u_i,v_i)\}_{i=1}^{n-1}$, sequence $(\alpha_1,\alpha_2,\ldots,\alpha_{n-1})$ and parameters $(H_n;R_{n-1};A_{n-1},B_{n-1})$. 

Assume that $v_iu_{i+1}\notin E(T)$, for $i=1,2,n-2$. Then,

\begin{itemize}
\item[(i)] the tree $T_1$ of type $H$ with parameters $(h_1,l_2;r_1;a_1,b_1)$, where $l_2=h_2$ if $\alpha_2\in \{\pi,\beta,\sigma^+\}$, and $l_2=h_2-a_2$, otherwise, is a path-like tree, with vertical path $\alpha_1$, horizontal divisors
 $d_1$ and $d_2$, respectively, and vertical divisor $\delta^1$,

\item[(ii)] for every $i\in \{2,3,\ldots, n-1\}$, the tree $T_i$ of type $H$ with parameters either $(d_i+a_i,l_{i+1};r_i;a_i,b_i)$  if $\alpha_i\in \{\pi,\gamma,\sigma^-,\mu\}$,  or, $(d_i,l_{i+1};r_i;a_i,b_i)$  if $\alpha_i\in \{\beta,\sigma^+\}$, where $l_{i+1}=h_{i+1}$ if either $i=n-1$ or $i<n-1$ and $\alpha_{i+1}\in \{\pi,\beta,\sigma^+\}$, and $l_{i+1}=h_{i+1}-a_{i+1}$, otherwise, is a path-like tree, with vertical path $\alpha_i$, horizontal divisors $d_i$ and $d_{i+1}$, respectively, and vertical divisor $\delta^i$.

\end{itemize}


\end{theorem}

\begin{proof}
Let $P^1_{k_1}, P^2_{k_2}, \ldots,P^m_{k_m}$ be the consecutive paths defined by a linear configuration of $T$. 


If $\alpha_2\in\{\pi, \beta, \sigma^+\}$ then by removing the edge incident to $u_2$ that connects $u_2$ to $v_2$, we obtain two path-like trees $T_1$ and $T_1'$.  $T_1$ is a tree of type $H$ with parameters $(h_1, h_2; r_1; a_1,b_1)$, vertical path $\alpha_1$, horizontal divisors $d_1$ and $d_2$, respectively and vertical divisor $\delta^1$.  
If $\alpha_2\in \{\gamma,\sigma^-,\mu\}$ and $u_2\in P^s_{k_{s}}$, then, by removing the edge that connects $P^s_{k_{s}}$ to $P^{s+1}_{k_{s+1}}$ of the linear configuration, and 
the subpath of  $P^{s}_{k_{s}}$  with end vertices $v^s_{k_{s}}$ and
 $u_2$, 
we also obtain two path-like trees: $T_1$ and $T_1'$.
$T_1$ is a tree of type $H$ with parameters $(h_1, h_2-a_2+d_2; r_1; a_1,b_1)$, vertical path $\alpha_1$, horizontal divisors $d_1$ and $d_2$, respectively and vertical divisor $\delta^1$. Notice that, since we assume that $v_i$ is not adjacent to $u_{i+1}$, for $i=1,2,n-2$, when $\alpha_2\in \{\gamma,\sigma^-,\mu\}$, we can obtain the divisors of $T_1$ by considering a path-like tree of type $H$ with parameters  $(h_1, h_2-a_2; r_1; a_1,b_1)$ and vertical path (or edge) $\alpha_1$.

In general, for every pair of vertices $\{u_i,v_i\}$ of degree $3$, where $i\ge 2$ and  $u_i\in P^t_{k_{t}}$, if we delete either the edge connecting $P^{t-1}_{k_{t-1}}$ with $P^{t-2}_{k_{t-2}}$ when
$\alpha_i\in\{\pi,\gamma,\sigma^-,\mu\}$, or the edge connecting $P^{t}_{k_{t}}$ with $P^{t-1}_{k_{t-1}}$ when $\alpha_i\in\{\beta,\sigma^+\}$, we obtain two path-like trees. Let's call $\overline{T_i}$ the one that contains $P^t_{k_{t}}$, which is of type $1$. We now repeat the construction provided above, that is:


If $\alpha_{i+1}\in\{\pi, \beta, \sigma^+\}$ then by removing the edge incident to $u_{i+1}$ that connects $u_{i+1}$ to $v_{i+1}$, we obtain two path-like trees: $T_i$ and $T_i'$. $T_i$ is a tree of type $H$ with a vertical path (or edge) $\alpha_i$, parameters either $(d_i+a_i, h_{i+1}; r_i; a_i,b_i)$ when $\alpha_{i}\in\{\pi, \gamma, \sigma^-,\mu\}$, or $(d_i,h_{i+1}; r_i; a_i,b_i)$ when $\alpha_{i}\in\{\beta, \sigma^+\}$, horizontal divisors $d_i$ and $d_{i+1}$, respectively and vertical divisor $\delta^i$.  
If $\alpha_{i+1}\in \{\gamma,\sigma^-,\mu\}$ and $u_{i+1}\in P^s_{k_{s}}$, then, by removing the edge that connects $P^s_{k_{s}}$ to $P^{s+1}_{k_{s+1}}$ of the linear configuration, and 
the subpath of  $P^{s}_{k_{s}}$  with end vertices $v^s_{k_{s}}$ and $u_{i+1}$, we also obtain two path-like trees: $T_i$ and $T_i'$. $T_i$ is a tree of type $H$ with a vertical path (or edge) $\alpha_i$, parameters either $(d_i+a_i,h_{i+1}-a_{i+1}+d_{i+1}; r_i; a_i,b_i)$ when $\alpha_{i}\in\{\pi, \gamma, \sigma^-,\mu\}$, or $(d_i,h_{i+1}-a_{i+1}+d_{i+1}; r_i; a_i,b_i)$ when $\alpha_{i}\in\{\beta, \sigma^+\}$, with horizontal divisors $d_{i}$, $d_{i+1}$, respectively and vertical divisor $\delta^{i}$. Notice that, since we assume that $v_i$ is not adjacent to $u_{i+1}$, for $i=1,2,n-2$, when $\alpha_{i+1}\in \{\gamma,\sigma^-,\mu\}$, we can obtain the divisors of $T_i$ by considering $h_{i+1}-a_{i+1}$ instead of $h_{i+1}-a_{i+1}+d_{i+1}$.

\end{proof}

The next result characterizes which trees in $\mathcal{H}_n$ are path-like trees, under the assumption that in each horizontal path there are not two adjacent vertices of degree three.

\begin{theorem}\label{theo: a tree in H_n such that...Then path-like tree_good}
Let $T$ be a tree in $\mathcal{H}_n$ with parameters $(H_n;R_{n-1};A_{n-1},B_{n-1})$ with no horizontal path that contains two adjacent vertices of degree $3$. Suppose that the following conditions hold.
\begin{itemize}
\item[(i)] The tree $T_1$ of type $H$ with parameters  $(h_1,t_2;r_1;a_1,b_1)$, where $t_2$ is either $h_2$ or $h_2-a_2$ is a path-like tree with vertical path $\alpha_1$, horizontal divisors $d_1$ and $d_2$, and vertical divisor $\delta^1$.

\item[(ii)] For every $i\in \{2,3,\ldots, n-1\}$, the tree $T_i$ of type $H$, with parameters $(s_{i},t_{i+1}; r_{i};\\ a_{i},b_{i})$, where $t_{i+1}$ is  either $h_{i+1}$ or  $h_{i+1}-a_{i+1}$ is a path-like tree with vertical path $\alpha_{i}$, horizontal divisors $d_{i}$ and $d_{i+1}$, respectively, vertical divisor $\delta^{i}$, where $s_{i}\in\{d_{i}, d_{i}+a_{i}\}$ is defined according to the following rules:

\begin{itemize}
\item if $t_{i}=h_{i}$, then either $s_{i}=d_{i}+a_{i}$ and $\alpha_{i}=\pi$; or $s_{i}=d_{i}$ and $\alpha_{i}\in\{\beta,\sigma^+\}.$
\item if $t_{i}=h_{i}-a_{i}$, then $s_{i}=d_{i}+a_{i}$ and $\alpha_{i}\in\{\gamma,\mu, \sigma^-\}.$
\end{itemize}

\item[(iii)] For every $2\le i\le n-1$,
$$d_i=\left\{\begin{array}{ll}
b_{i-1}, & \hbox{if } \alpha_{i-1}=\pi\\
a_{i-1}+r_{i-1}, & \hbox{if } \alpha_{i-1}=\gamma, i\neq 2\\
a_{1}+r_{1}, \hbox{ or }  h_1-a_1+1+r_{1},& \hbox{if } \alpha_{1}=\gamma, i=2\\
\delta^{i-1}+b_{i-1} -1,& \hbox{if } \alpha_{i-1}=\beta,\sigma^-, \delta^{i-1}\neq b_{i-1}\\
h_{i},& \hbox{if } \alpha_{i-1}=\beta,\sigma^-, \delta^{i-1}= b_{i-1}\\
\delta^{i-1}-b_{i-1} +1,& \hbox{if } \alpha_{i-1}=\mu,\sigma^+\\
\end{array}\right.
$$
and
$$
 d_i=\left\{\begin{array}{ll}
a_i, & \hbox{if } \alpha_{i}=\pi\\
b_i+r_i, & \hbox{if } \alpha_{i}=\gamma, i\neq n-1\\
b_{n-1}+r_{n-1},  \hbox{ or } h_{n}-b_{n-1}+1+r_{n-1},  & \hbox{if } \alpha_{n-1}=\gamma, i= n-1\\
\delta^{i}+a_{i} -1,& \hbox{if } \alpha_{i}=\beta,\sigma^+, \delta^{i}\neq a_{i} \\
h_{i},& \hbox{if } \alpha_{i}=\beta,\sigma^+, \delta^{i}= a_{i}\\
\delta^{i}-a_{i} +1,& \hbox{if } \alpha_{i}=\mu,\sigma^-\\
\end{array}\right.$$
\item[(iv)] And, whenever $\alpha_i\in\{\beta,\sigma^+\}$, if either $\alpha_{i-1}\in \{\pi,\gamma,\mu,\sigma^+\}$ or, $\alpha_{i-1}\in \{\beta,\sigma^-\}$ and $d_{i}<h_i$ then $d_i=\delta^{i}+a_{i} -1$. Otherwise, if $\alpha_{i-1}\in \{\beta,\sigma^-\}$ and $d_{i}=h_i$ then, either $\delta^{i-1}=b_{i-1}$ and $\delta^i=a_i$ or $\delta^{i-1}=h_i-b_{i-1}+1$ and $\delta^i=h_i-a_i+1$.
\end{itemize}
Then $T$ is a path-like tree.
\end{theorem}

\begin{proof}
The proof is a constructive proof. We will show that if there exists such a sequence of path-like trees satisfying the ``glue conditions'' in (iii) and (iv), we can obtain a linear configuration of $T$, which implies by Theorem \ref{theo: plt iff a linear config. exists} that $T$ is a path-like tree.

We start by obtaining a linear configuration of a subgraph of $T$ with parameters $(H'_3;R_{2};A_{2},B_{2})$, where, $h_j'=h_j$, $j=1,2$ and $h_3'=t_3.$ We will denote such a tree by $T_2^*$.
Let $P^{1,1}_{k_1}, P^{1,2}_{k_2},\ldots,P^{1,m}_{k_m}$ and $P^{2,1}_{k'_1}, P^{2,2}_{k'_2},\ldots,P^{2,m'}_{k'_{m'}}$ be the paths defined by a linear configuration of $T_1$ and $T_2$, respectively, with $|V(P^{1,m}_{k_m})|=|V(P^{2,1}_{k'_1})|=d_2$, that is, $k_m=k'_1=d_2$. We let $P^{1,l}_{k_l}=v^{1,l}_{1}v^{1,l}_{2}\cdots v^{1,l}_{k_l}$ and similarly, $P^{2,l}_{k'_l}=v^{2,l}_{1}v^{2,l}_{2}\cdots v^{2,l}_{k'_l}$.

\underline{Case $\alpha_2\in\{\pi,\gamma\}$.} We have that $u_2\in P^{2,2}_{k'_2}$ with $v^{2,2}_{d_2}=u_2$.

\begin{itemize}
\item \underline{Subcase $\alpha_1\in \{\pi,\gamma, \mu,\sigma^+\}$.} Since we assume that there are no adjacent vertices of degree $3$ in each horizontal path of $T$, when $\alpha_2=\pi$, we have that $v_1\in P^{1,s}_{k_{s}}$, for some $s<m-1$ and that $|V(P^{1,m-1}_{k_{m-1}})|=|V(P^{1,m}_{k_m})|=d_2$. Thus, by identifying either the paths $P^{1,m-1}_{k_{m-1}}$, $P^{1,m}_{k_m}$ with the path $P^{2,1}_{k'_1}$ and the subpath of $P^{2,2}_{k'_2}$ defined by $v^{2,2}_1\ldots v^{2,2}_{d_2}$, respectively, when $\alpha_2=\pi$; or the path $P_{k_m}^{1,1}$ with $P_{k'_1}^{2,1}$, when $\alpha_2=\gamma$, we obtain a linear configuration of a tree $T_2^*$ in $\mathcal{H}_3$, which in view of (iii) has parameters $((h_1,h_2,t_3);(r_1,r_2);(a_1,a_2),(b_1,b_2))$.
\vspace{0.5cm}
\item \underline{Subcase $\alpha_1=\beta, \sigma^-$.} Notice that, since $d_2=a_2$ when $\alpha_2=\pi$,
we have that $v_1\in P^{1,s}_{k_{s}}$, for some $s<m$ and we also have that $|V(P^{1,m-1}_{k_{m-1}})|=|V(P^{1,m}_{k_m})|=d_2$. Thus, by identifying either the paths $P^{1,m-1}_{k_{m-1}}$, $P^{1,m}_{k_m}$ with the path $P^{2,1}_{k'_1}$ and the subpath of $P^{2,2}_{k'_2}$ defined by $v^{2,2}_1\ldots v^{2,2}_{d_2}$, respectively, when $\alpha_2=\pi$; or the path $P_{k_m}^{1,1}$ with $P_{k'_1}^{2,1}$, when $\alpha_2=\gamma$, we obtain a linear configuration of a tree $T_2^*$ in $\mathcal{H}_3$, which in view of (iii) has parameters $((h_1,h_2,t_3);(r_1,r_2);(a_1,a_2),(b_1,b_2))$.
\end{itemize}

\vspace{0.5cm}

\underline{Case $\alpha_2\in\{\beta,\sigma^+\}$.} By condition (iii), we have that $u_2\in P^{2,1}_{k'_1}$ with either $v^{2,1}_{a_2}=u_2$ or  $v^{2,1}_{h_2-a_2+1}=u_2$.
By identifying the path $P^{1,m}_{k_m}$ with the path $P^{2,1}_{k'_1}$ in such a way that the vertex $v^{1,m}_{1}$ corresponds to vertex $v^{2,2}_1$, we obtain a linear configuration of a tree $T_2^*$ in $\mathcal{H}_3$, which in view of (iii)  and (iv) has parameters $$((h_1,h_2,t_3);(r_1,r_2);(a_1,a_2),(b_1,b_2)).$$

\underline{Case $\alpha_2\in\{\mu,\sigma^-\}$.} By condition (iii), we have that $u_2\in P^{2,2}_{k_2}$. By identifying the path $P^{1,m}_{k_m}$ with the path $P^{2,1}_{k_1}$ in such a way that the vertex $v^{1,m}_{1}$ corresponds to vertex $v^{2,2}_1$, we obtain a linear configuration of a tree $T_2^*$ in $\mathcal{H}_3$, which in view of (iii) has parameters $$((h_1,h_2,t_3);(r_1,r_2);(a_1,a_2),(b_1,b_2)).$$

Thus, assume that we have obtained a linear configuration of a tree $T_i^*$ in $\mathcal{H}_{i+1}$, with parameters $(H'_{i+1};R_{i};A_{i},B_{i})$, where $h_j'=h_j$, for $j=1,2,\ldots, i$ and $h_{i+1}'=t_{i+1}$. Let $P^{1,1}_{k_1}, P^{1,2}_{k_2},\ldots,P^{1,m}_{k_m}$ and $P^{2,1}_{k'_1}, P^{2,2}_{k'_2},\ldots,P^{2,m'}_{k'_{m'}}$ be the paths defined by a linear configuration of $T_i^*$ and $T_{i+1}$, respectively, with $|V(P^{1,m}_{k_m})|=|V(P^{2,1}_{k'_1})|=d_i$, that is, $k_m=k'_1=d_i$. By adapting the process explained above, we get a linear configuration of a subgraph $T_{i+1}^*$ of $T$, which is of type $\mathcal{H}_{i+2}$, with parameters $(H'_{i+2};R_{i+1};A_{i+1},B_{i+1})$, where $h_j'=h_j$, for $j=1,2,\ldots, i+1$ and $h_{i+2}'=t_{i+2}$. Hence, repeating this process until we get $i=n-2$, we obtain a linear configuration of $T$. Therefore, by Theorem \ref{theo: plt iff a linear config. exists}, $T$ is a path-like tree.
\end{proof}

\subsection{Path-like trees of maximum degree $4$}\label{section: degree 4}
The study of path-like trees of degree $4$ can be reduced to the study of path-like trees of maximum degree $3$. The next lemma shows how this reduction can be performed. We denote by $\subgp{H}$ the subgraph induced by $H$.

\begin{lemma}
Let $T$ be a tree with a vertex $v$ of degree $4$. Denote by $T'_1, T'_2,T'_3$ and $T'_4$ the connected components of $T-v$. Then, $T$ is a path like tree if and only if there exists a permutation $\varphi\in \mathcal{S}_4$ such that,
\begin{itemize}
\item[(i)] $T_1=\subgp{T_{\varphi (1)}\cup\{v\}\cup T_{\varphi (2)}}$ and $T_2=\subgp{T_{\varphi (3)}\cup\{v\}\cup T_{\varphi (4)}}$ are path-like trees, and
\item[(ii)]  there exist a linear configuration of $T_1$ and a linear configuration of $T_2$, in which $v_1^{1,1}v_2^{1,1}\ldots v_{k_1}^{1,1}$ and $v_1^{2,1}v_2^{2,1}\ldots v_{k'1}^{2,1}$ are the first paths defined by the linear configuration of $T_1$ and $T_2$, respectively, such that $v_1^{1,1}=v$ and $v=v_1^{2,1}$.
\end{itemize}

\end{lemma}
\section{Conclusion}
Path-like trees is a family of trees of maximum degree four, which has good properties in terms of labelings and also in terms of decompositions. In this paper, we have defined a new way of introducing them. Using this new way, we have characterized which trees of type $\mathcal{H}_{n}$, with no adjacent vertices of degree $3$ in each horizontal path are path-like trees. The problem that is still open consists on characterizing path-like trees in general, that is, without the restriction on the number of vertices of degree 3 or, on the non adjacency of vertices of degree $3$ that are in a horizontal path. It seems that these two constraints can be attacked by considering trees of type cutted $H$ (see Section \ref{subsection: cutted $H$}). However, it is not clear how we can associate the set of divisors (key point in the proof of Theorem \ref{theo: a tree in H_n such that...Then path-like tree_good}) in that case.

\section*{Acknowledgements}
The research conducted in this document by the first author  has been supported by the Spanish Research Council under project
MTM2014-60127-P and symbolically by the Catalan Research Council under grant 2014SGR1147.



\end{document}